\documentclass[10pt,a4paper]{amsart}
\usepackage[final]{optional}
\usepackage[british,american]{babel}
\usepackage{amsfonts, amsmath, wasysym}
\usepackage{verbatim}

\usepackage{amssymb,amsthm,
paralist
}

\usepackage{
latexsym,
nicefrac,
}

\usepackage[usenames]{color}

\usepackage{url}

\definecolor{darkgreen}{rgb}{0,0.5,0}
\definecolor{darkred}{rgb}{0.7,0,0}
\usepackage[colorlinks, 
citecolor=darkgreen, linkcolor=darkred
]{hyperref}
\usepackage{esint}
\usepackage{bibgerm}
\usepackage[normalem]{ulem}

\theoremstyle{plain}
\newtheorem{lemma}{Lemma}[section]
\newtheorem{thm}[lemma]{Theorem}
\newtheorem{prop}[lemma]{Proposition}
\newtheorem{cor}[lemma]{Corollary}
\newtheorem{remark}[lemma]{Remark}

\theoremstyle{definition}
\newtheorem{defn}[lemma]{Definition}

\newtheorem{rmk}[lemma]{Remark}
\newtheorem{Prop}[lemma]{Proposition}

\numberwithin{equation}{section}

\newcommand{\de}{\delta}

\newcommand{\Om}{\Omega}

\renewcommand{\th}{\theta}
\newcommand{\Th}{\Theta}

\newcommand{\R}{\ensuremath{{\mathbb R}}}
\newcommand{\N}{\ensuremath{{\mathbb N}}}

\newcommand{\beq}{\begin{equation}}
\newcommand{\eeq}{\end{equation}}
\newcommand{\beqs}{\begin{equation*}}
\newcommand{\eeqs}{\end{equation*}}
\newcommand{\beqa}{\begin{equation}\begin{aligned}}
\newcommand{\eeqa}{\end{aligned}\end{equation}}
\newcommand{\beqas}{\begin{equation*}\begin{aligned}}
\newcommand{\eeqas}{\end{aligned}\end{equation*}}
\newcommand{\brmk}{\begin{rmk}}
\newcommand{\ermk}{\end{rmk}}
\newcommand{\partref}[1]{\hbox{(\csname @roman\endcsname{\ref{#1}})}}

\newcommand*\ddt{\frac{\mathrm{d}}{\mathrm{d}t}}

\newcommand{\pt}{\partial_t}
\newcommand{\M}{\ensuremath{{\mathcal M}}_{-1}}
\newcommand{\A}{\ensuremath{{\mathcal A}}}
\newcommand{\D}{\ensuremath{{\mathcal D}}}
\newcommand{\dxi}{\partial_{x_i}}
\newcommand{\dxj}{\partial_{x_j}}
\newcommand{\dds}{\frac{d}{ds}}
\newcommand{\norm}[1]{\Vert#1\Vert} 
\newcommand{\abs}[1]{\left\vert#1\right\vert} 
\newcommand{\seminorm}[1]{[#1]} 
\newcommand{\eps}{\varepsilon}
\newcommand{\na}{\nabla}
\newcommand*\tr{\mathop{\mathrm{tr}}\nolimits}

\begin{document}
\title[Flowing to minimal surfaces: Existence and uniqueness]{Flowing maps to minimal surfaces: Existence and uniqueness of solutions} 
\author{Melanie Rupflin}
\date{\today}


\begin{abstract}
We study the new geometric flow that was introduced in \cite{R-T} that evolves a pair of map and (domain) metric in such a way that it changes appropriate initial data into branched minimal immersions. 
In the present paper we focus on the existence theory as well as the issue of uniqueness of solutions. We establish that a (weak) solution exists for as long as the metrics remain in a bounded region of
moduli space, i.e.~as long as the flow 
does not collapse a closed geodesic in the domain manifold to a point. Furthermore, we prove that this solution is unique in the class of all weak solutions with non-increasing energy. This work complements the paper \cite{R-T} of Topping and 
the author where the flow was introduced and
its asymptotic convergence to branched minimal immersions is discussed. 
\end{abstract}
\maketitle
\section{Introduction}
Let $M$ be a smooth closed orientable surface and let $(N,G_N)$ be a (fixed) closed smooth Riemannian manifold of arbitrary dimension 
that we view as being isometrically immersed in $\R^K$ for some $K\in \N$. 

For $g$ a Riemannian metric on $M$ and a map $u:(M,g)\to (N,G_N)$ the Dirichlet energy is defined as
\beqs
E(u,g):=\frac12\int_M |du|^2\,dv_g.
\eeqs 
We remark that $(u,g)$ is a critical point of $E$ if and only if $u$ is harmonic and weakly conformal, i.e.~a branched minimal immersion or a constant map. 
In the present paper we establish the existence theory for the natural gradient flow of $E$ (considered as a function of both the map and the domain metric) which was introduced in \cite{R-T}. 
We refer to this joint paper of Topping and the author for the construction and the geometric background of this flow, but for convenience here 
recall the main points that led to the definition in \cite{R-T}. 

We consider the negative gradient flow of $E$ considered as a function of both the map and the domain metric, but taking into account the symmetries of $E$, that is the invariance under conformal variations of the domain as well as under the pull-back by diffeomorphisms applied simultaneously to the metric and the map 
component. That is we consider $E$ and its gradient flow on the set 
$$\A=\{[(u,g)];\, g\in \mathcal{M}_c, u\in C^\infty(M,N)\}$$ of equivalence classes where we identify  
$(u,g)\sim (u\circ f,f^*g)$ for smooth diffeomorphisms $f:M\to M$ homotopic to the identity. 
Here $\mathcal{M}_c$ stands for the set of smooth metrics of constant 
(Gauss-)curvature $c=1,0,-1$ for surfaces of genus $\gamma=0,1$ respectively $\gamma\geq 2$, with unit area in case $\gamma=1$. 

The tangent space of $\mathcal{M}_c$ splits orthogonally into a horizontal part consisting of the real parts of holomorphic quadratic differentials and a vertical part along the 
fibers of the action of diffeomorphisms on $\mathcal{M}_c$, i.e.~the space of Lie-derivatives of the metric, compare Lemma \ref{lemma:horiz} below. 
This canonical splitting allows us in \cite{R-T} to represent solutions of the $L^2$-negative gradient flow of $E$ on $\A$ by the solutions of the system 
\begin{subequations}\label{1}
  \begin{align}
       \pt u&=\tau_g(u) \label{1a}\\
     \frac{d g}{dt}& =\frac{\eta^2}4\text{Re}(P^H_g(\Phi(u,g))).\label{1b}
  \end{align}
\end{subequations}

Here $\tau_g(u)=\text{tr}_g(\na_g(du))=\Delta_g u+A_g(u)(\na u,\na u)$, $A$ the second fundamental form of $N\hookrightarrow \R^K$,
denotes the tension field of $u:(M,g)\to (N,G_N)$ and $\Phi(u,g)$ stands for the Hopf-differential, i.e.~the quadratic differential given in conformal coordinates $z=x+iy$ of $(M,g)$ 
as $\Phi(u,g)=\phi dz^2$ for $\phi=\abs{u_x}^2-\abs{u_y}^2-2i\langle u_x,u_y\rangle$. Furthermore 
$P_g^H$ denotes the $L^2$-orthogonal projection from the space of quadratic differentials onto the finite dimensional subspace of \emph{holomorphic} quadratic differentials on $(M,g)$. 
Finally $\eta>0$ is a free coupling constant related to the choice of $L^2$-metric on $\A$. 

As the main result of this paper we prove the following existence and uniqueness theorem
\begin{thm}\label{thm:1}
To any given initial data $(u_0,g_0)\in  C^\infty(M,N)\times \mathcal{M}_c$ there exists a weak solution $(u,g)$ of \eqref{1} defined on a maximal interval $[0,T)$, $T\leq \infty$, that satisfies
the following properties
\begin{enumerate}
\item[(i)] The solution $(u,g)$ is smooth away from at most finitely many singular times $T_i\in(0,T)$ at which `harmonic spheres bubble off'. 
More precisely as $t\nearrow T_i$ energy concentrates at a finite number of points $S(T_i)\subset M$ and suitable rescalings of the maps $u(t)$ around points in $S(T_i)$ converge as 
$t\nearrow T_i$ to (a bubble-tree of) non-trivial harmonic maps from $\R^2\cup \{\infty\}\cong S^2$ to $N$. 
\item[(ii)] As $t\to T_i$ the maps $u(t)$ converge weakly in $H^1$ and smoothly away from the set $S(T_i)$ to a limit $u(T_i)\in H^1(M,N)$. 
Furthermore, the metrics $g(t)$ converge smoothly to an element $g(T)\in \mathcal{M}_c$; in fact, the flow of metrics is Lipschitz-continuous with respect to all $C^m$ metrics on $\mathcal{M}_c$
across singular times.
\item[(iii)] The energy $t\mapsto E(u(t),g(t))$ is non-increasing. 
\item[(iv)] The solution exists as long as the metrics do not degenerate in moduli space; i.e.~either $T=\infty$ or the length $\ell(g(t))$ of the shortest closed geodesic in $(M,g(t))$ converges to zero as 
$t\nearrow T$.
\end{enumerate} 
Furthermore, the solution is uniquely determined by its initial data in the class of all weak solutions with non-increasing energy. 
\end{thm}

\begin{defn}\label{def:weak}
We call $(u,g)\in H^1_{loc}(M\times[0,T),N)\times C^0([0,T),\M)$ a weak solution of \eqref{1} if 
$u$ solves \eqref{1a} in the sense of distributions 
and if $g$ is piecewise $C^1$ (viewed as a map from $[0,T)$ into the space of symmetric $(0,2)$ tensors equipped with any $C^k$ metric, $k\in\N$) 
and satisfies \eqref{1b} away from times where it is not differentiable.
\end{defn}

We remark that the assumption on the initial data in Theorem \ref{thm:1} can be weakened to $u_0\in H^1(M,N)$ with the resulting solution being smooth away from finitely many times, possibly including $T_1=0$.

On intervals where the obtained solution $(u,g)$ is smooth, the energy decays by
\beq \label{energy-identity}
\ddt E(u(t),g(t))=-\int_M \abs{\tau_g(u)}^2 dv_g-\frac{\eta^2}{16} \norm{Re[P_g^H(\Phi(u,g))]}_{L^2(M,g(t))}^2,
\eeq
so that if $T=\infty$, both the tension field as well and the holomorphic part of the Hopf-differential converge to zero as 
$t\to \infty$ suitably. In the joint paper \cite{R-T} of Topping and the author we indeed prove that if the metric does not degenerate even as $t\to \infty$, then the full Hopf-differential 
(sub)converges to zero, resulting in a limit that is both harmonic and weakly conformal and thus, if non-constant, a minimal immersion away from at most finitely many branch-points \cite{GOR}. More precisely, 
in \cite{R-T}, we prove

\begin{thm}[\cite{R-T}, Thm 1.4]
In the setting of Theorem \ref{thm:1}, if the length $\ell(g(t))$ of the shortest closed geodesic of $(M,g(t))$ is uniformly bounded below by a positive constant, 
then there exist a sequence of times $t_i\to\infty$ and  a sequence of orientation-preserving diffeomorphisms $f_i:M\to M$ such that 
$$f_i^*g(t_i)\to \bar g \,\text{ and } \,  u(t_i)\circ f_i \to \bar u$$
converge to a metric $\bar g\in \mathcal{M}_c$ and a branched minimal immersion $\bar u$ or a constant map. 
Here the convergence of metrics is smooth, while the maps converge weakly in $H^1(M,N)$ and strongly in $W^{1,p}_{loc}(M\backslash S)$ for any $p\in [1,\infty)$ away from a finite set of points where
energy concentrates. 
\end{thm}

For suitable initial data, such as incompressible maps, a degeneration of metrics can be excluded so that the flow (sub)converges (up to reparametrisations)
to a branched minimal immersion. In \cite{R-T} we thus recover the 
well known results on the existence of branched minimal immersions with given action on the level of fundamental groups of Schoen-Yau \cite{Schoen-Yau} and Sacks-Uhlenbeck \cite{Sacks-Uhlenbeck}
with a flow approach. 

Solutions of \eqref{1} that degenerate in moduli space will be analysed in a forthcoming paper \cite{Rup-Top-Zhu} by Topping, Zhu and the author.

\begin{remark}
For surfaces of genus less than two the structure of the flow \eqref{1} is simplified considerably and the existence of solutions is known; for spheres the space 
of holomorphic quadratic differentials is trivial so \eqref{1} reduces to the harmonic map flow of Eells and Sampson \cite{Eells+Sampson}  for which existence of global weak solutions was 
proven in the seminal 
paper of Struwe \cite{Struwe85}. For maps from a surface of genus $1$ it is shown in \cite{R-T} that \eqref{1} agrees with a flow that was introduced and studied by 
Ding, Li and Liu in \cite{Ding-Li-Liu}. In this special case the flow of metrics is reduced to two scalar ODEs for parameters describing a global \emph{horizontal} submanifold of 
the space of metrics. 
Furthermore, the completeness of Teichm\"uller space prevents a degeneration of the metric at finite times, leading to the existence of global (weak) solutions for all initial data as obtained
in \cite{Ding-Li-Liu}.
\end{remark}

In this paper we thus focus on the analysis of the flow from general surfaces of genus $\gamma\geq 2$.

\textbf{Outline of the paper}\\
The paper consists of three main parts. In the first section we study the properties of \textit{horizontal curves}, i.e.~curves that move in the direction of the real part of holomorphic quadratic differentials. 
Using ideas from Teichm\"uller theory, we obtain strong estimates for all horizontal curves, and thus in particular for the metric component of the flow, under the sole 
condition that we stay away from the boundary of moduli space.

In the second section we prove the existence of solutions as claimed in Theorem \ref{thm:1}. First we obtain short-time existence of smooth solutions based on the properties of horizontal curves 
derived in the first section. 
In a second step we then analyse the possible finite time singularities of the flow. On the one hand, we prove that the only way for the metric component to become singular is by a degeneration in moduli space. On the other hand, we obtain that as long as the metric component remains regular, the behaviour of solutions to \eqref{1a}
is similar to the one of solutions of the harmonic map flow as described by Struwe in \cite{Struwe85}; the singularity is caused by the \textit{bubbling off} of harmonic spheres
and the flow can be continued past the singular time by a weak solution. 

Finally we consider the question of uniqueness. We show uniqueness not just for solutions of \eqref{1} satisfying properties (i)-(iii) of Theorem \ref{thm:1} but in the general class of weak solutions with non-increasing energy. 
This represents the analogue of the uniqueness results \cite{Freire2} and \cite{Freire} of Freire for the harmonic map flow.
\begin{rmk}\label{rmk:no-symmetries}
For general curves within $\M$, satisfying an $L^2$ bound on the velocity such as \eqref{energy-identity}, 
singularities can form without the metrics degenerating in moduli space. For the flow 
\beq \label{2} \pt u=\tau_g(u),\quad \frac{d g}{dt}=\frac{\eta^2}4 \text{Re}(\Phi(u,g)),\eeq
which we would obtain if we were to consider the gradient flow of $E$ without taking 
into account the symmetries, we thus would not have a characterisation of the maximal existence time of solutions as statement (iv) of Theorem \ref{thm:1}. 
For \eqref{2} we thus could not expect to obtain the global solutions needed to evolve pairs $(u,g)$ to critical points of the energy, i.e.~to branched minimal immersions,
even for incompressible initial data.
\end{rmk}

Acknowledgements: The author thanks Peter Topping for valuable discussions. This work was partially supported by The Leverhulme Trust. 

\section{Horizontal curves}\label{section:horizontal}
We consider general horizontal curves, that is curves moving in the direction 
$$\ddt g=\text{Re}(\Psi(t))$$
of holomorphic quadratic differentials $ \Psi(t)=\psi(t) dz^2$ on $(M,g(t))$, $z=z(t)$ a complex coordinate on $(M,g(t))$.
Key for the analysis of such curves is a good understanding of the dependence on the metric $g\in\M$ of the horizontal space 
$$H(g):=\{\text{Re}(\Phi):\ \Phi=\phi dz^2 \, \text{holomorphic quadratic differential on } (M,g)\}$$
and of the corresponding $L^2$-orthogonal projection $P_g^H$. 
What we essentially need is a quantified version of the idea that a smooth variation of the metric leads to a smooth variation of the complex structure, which in 
turn results in a smooth change of the space of holomorphic
quadratic differentials and of $P_g^H$.

We remark that there are several equivalent points of view that one can take to study horizontal tensors and curves as well as to study the flow \eqref{1}. Here we follow the differential geometrical 
approach to Teichm\"uller theory as presented in the book of Tromba \cite{Tromba}. 
We view the space of horizontal tensors 
as a subspace $H(g)$ of the space $Sym^2(M)$ of all real symmetric $(0,2)$ tensors of class $L^2$, 
with $H(g)$ characterised by 
\beqs H(g)=\{h\in Sym^2(M):\, \tr_g(h)=0\text{ and } \delta_g h=0 \},\eeqs
$\delta_g$ the divergence operator (induced by the Levi-Civit\`a connection $\na_g$).
We then consider the projection 
$$P_g:Sym^2(M)\to H(g)$$ 
that is orthogonal with respect to the $L^2(M,g)$-inner product
\beqs 
\langle k,h\rangle_{L^2(M,g)}:=\int_M g^{ij} g^{lm}k_{il}h_{jm} \,dv_g.\eeqs 
This projection $P_g$, for which we shall derive an explicit formula later on, is related to the projection $P_g^H$ from the space of quadratic differentials to the space of holomorphic quadratic differentials by
\beq \label{eq:PgH}P_g(Re(\psi dz^2))=Re(P_g^H(\psi dz^2))\eeq
for any quadratic differential $\psi dz^2$ on $(M,g)$.

We first remark that the set of hyperbolic metrics $\M$ with smooth coefficients (in the given coordinate charts) is not a manifold. 
On the other hand for any $s>3$ the set $\M^s$ of hyperbolic metrics with coefficients in the Sobolev space $H^s(M)$ 
is a smooth submanifold of the (half-)space of all $H^s$ metrics on $M$, see \cite{Tromba} Theorem 1.6.1. 

We shall think of the projection $P_g$ as a map from this Banach manifold $\M^s$ into the space $L(Sym^2(M),T\M^s$) of linear functions 
mapping symmetric $(0,2)$-tensors into the tangent bundle of $\M^s$ and prove that it is locally Lipschitz.

\begin{prop}
\label{Prop:S}
For any smooth hyperbolic metric $g_0\in \M$ and every $s>3$ there exists a neighbourhood $W$ of $g_0$ in the Banach manifold $\M^s$ and a constant  $C=C(g_0,s)<\infty$ such
that the following holds true:
For every tensor $k\in Sym^2(M)$ and every curve $g\in C^1([0,T),\M^s)$ contained in $W$ 
we have 
\beq \label{est:Pg1} \norm{P_{g}(k)}_{H^s}\leq C\cdot \norm{k}_{L^2(M,g)}\eeq
and 
\beq \label{est:Pg2} \norm{\ddt P_{g(t)}(k)}_{H^s}\leq C\cdot \norm{\ddt g(t)}_{H^s}\cdot \norm{k}_{L^2(M,g)}.\eeq
\end{prop}
Here and in the following the Sobolev norms $\norm{\cdot}_{H^s}$ are to be computed in \textit{fixed} local coordinate charts of $M$.

Based on this local statement about the projection $P_g$, we then derive the following result for horizontal curves contained in compact regions of moduli space. 

\begin{prop}
\label{Prop:horizontal}
For every $\eps>0$ and every $s>3$ there exists a number $\th=\th(\eps,s)>0$ such that the following holds true. 
Let $g_0\in \M^s$ be any hyperbolic metric of class $H^s$ for which the length 
$\ell(g_0)$ of the shortest closed geodesic in $(M,g_0)$ is no less than $\eps$. Then there is a number $C=C(g_0,s)<\infty$ such that for any \textit{horizontal} curve 
$g\in C^1([0,T),\M^s)$ with $g(0)=g_0$ and of $L^2$-length 
$$L(g)=\int_0^T\norm{\ddt g(t)}_{L^2(M,g(t))} dt\leq \theta$$
we have
\beq
\label{est:Ck}
\norm{\frac{d}{dt} g(t)}_{H^s}\leq C\norm{\frac{d}{dt} g(t)}_{L^2(M,g(t))} \text{ for every } t\in [0,T).
\eeq
\end{prop}

For tori the corresponding result is obtained as a consequence of the existence of a \textit{smooth global horizontal slice}, i.e.~of a finite 
dimensional smooth submanifold of $\mathcal{M}_0$, parametrised over Teichm\"uller space, whose tangent space at each point is horizontal and which thus contains all horizontal curves passing through $g_0$. 

While for surfaces of genus $\gamma\geq 2$ the space of horizontal tensors $H(g)$ is still finite dimensional, $\dim_\R(H(g))=6\gamma-6$ by the Riemann-Roch theorem, the distribution
$g\mapsto H(g)$ is no longer integrable, compare \cite{Tromba}, section 5.3, so Proposition \ref{Prop:horizontal} cannot be reduced to a statement about curves on a finite dimensional manifold.

\begin{proof}[Proof of Proposition \ref{Prop:S}] 

We prove Proposition \ref{Prop:S} in two steps; we show first that estimates of the form \eqref{est:Pg1} and \eqref{est:Pg2} hold true for metrics contained in a so called \textit{slice}
and then in a second step pull-back these estimates to give the claim of Proposition \ref{Prop:S} for general metrics in a neighbourhood of $g_0$. To do so we make use of ideas from Teichm\"uller theory as explained in the book of Tromba \cite{Tromba}, chapter 2.

So let $g_0\in \M$ be any given metric and let $s>3$ be fixed. 
Following \cite{Tromba} we define a small \textit{slice} around $g_0$ by
\beq \label{eq:slice}
S:=\{g=\rho(h)\cdot(g_0+h):\, h\in U\subset H(g_0)\}\subset \M
\eeq
for $U=U(g_0,s)$ a suitably small neighbourhood of $0\in H(g_0)$ chosen later on. Here the function $\rho(h):M\to \R$ is to be chosen such that $\rho(h)\cdot (g_0+h)$ has constant 
curvature $-1$ and is uniquely determined by this property according to Poincar\'e's theorem. 

The key feature of this finite dimensional submanifold of $\M^s$ is that it provides a local model of $\M^s/\D_0^{s+1}$, 
with $\mathcal{D}_0^{s+1}$ the set of $H^{s+1}$-diffeomorphisms that are homotopic to the identity

\begin{thm}[\cite{Tromba}, Thm 2.4.3]\label{thm:Tromba}
For any number $s>3$, any $g_0\in \M$ and $S=S(g_0,s)$ a sufficiently small slice around $g_0$, there are neighbourhoods $W \subset \M^s$ of $g_0$ and 
$V\subset \mathcal{D}_0^{s+1}$ of $id$ for which 
the map 
$$S\times V\ni (g,f)\mapsto f^*g\in W$$ 
is a diffeomorphism.
\end{thm}

For a proof of this theorem as well as for further insight into Teichm\"uller theory we refer to the book of Tromba \cite{Tromba}. We remark that the above result remains valid
if we replace the slice $S$ by a smaller slice defined by \eqref{eq:slice}, for appropriate new neighbourhoods of $id$ in $\D_0^{s+1}$ and of $g_0$ in $\M^s$, but that the theorem
does not give the existence
of a uniform slice for which the statement is valid for all numbers $s>3$. We furthermore stress that the theorem demands that the metric $g_0$ is not only in $\M^s$ 
but smooth; this in turn implies that all metrics contained in a small 
slice $S$ are smooth and thus satisfy stronger estimates than just $H^s$ bounds, in particular

\begin{lemma} \label{lemma:S}
For a sufficiently small slice $S$ around $g_0\in \M$ there exists a constant $C=C(s,g_0)<\infty$ such that for all metrics $g_{1,2}\in S$  
\beq
\label{est:lemma:S}
\norm{g_1-g_2}_{H^{s+1}}\leq C\cdot d_S(g_1,g_2).
\eeq
\end{lemma}
Here we denote by $d_S$ the $H^s$ metric on $S$, i.e.~consider $S$ as a submanifold of the Banach manifold $\M^s$.

Apart from the finite dimensionality of $H(g_0)$, and thus of $S$, the essential observation leading to the above estimate is that the conformal factor $\rho(h)$ can be characterised 
as the unique solution of an elliptic 
PDE, compare \cite{Tromba} section 1.5, leading to a smooth dependence of $\rho(h)$ on $h\in U$.

Based on these stronger estimates on elements of the slice, we can analyse the dependence of $P_g$ on $g\in S$ using an explicit formula for $P_g$ 
that we shall derive now.

We first recall the following canonical splitting of the tangent space $T_g\M^s$ into the horizontal and vertical space, see Theorem 2.4.1 of \cite{Tromba}.
\begin{lemma} \label{lemma:horiz}
For any $g\in \M$ the tangent space $T_g\M^s$ splits $L^2$-orthogonally into $H(g)$ and the space $\{L_Xg\}$ of Lie-derivatives. 
More precisely, given any $k\in T_g\M^s$ there is a unique vector field $X$ (of class $H^{s+1}$) such that 
$$tr_g(k-L_Xg)=0 \text{ and } \delta_g(k-L_Xg)=0$$
and $X$ can be characterised as the unique solution of the elliptic PDE
\beq \label{eq:X}
\delta_g\delta_g^* X=-\delta_g k,\eeq
 $\delta_g^*X=-L_Xg$ the $L^2(M,g)$-adjoint of $\de_g$.
\end{lemma}

In order to define the orthogonal projection of a general symmetric $(0,2)$ tensor $k$ onto the horizontal space $H(g)$, we first map $k$ onto an element of $T_g\M^s$ using 
\begin{lemma}\label{lemma:Proj1}
For any $g\in \M$ and any symmetric $(0,2)$ tensor $k$ of class $H^s$ there exists a unique function $\mu\in H^s(M,\R)$ such that 
$$k-\mu \cdot g\in T_g\M^s.$$
The function $\mu$ is characterised as the unique solution of the equation 
\beq \label{eq:mu}
-\Delta_g \mu+2\mu=2DR(g)(k),
\eeq
$R(g)$ the Gauss curvature of $(M,g)$.
\end{lemma}

Given any $g\in \M$, we now claim that the orthogonal projection $P_g:Sym^2(M)\to H(g)$ is given by 
\beq \label{def:PH}
P_g(k):= k-\mu(k,g)\cdot g-L_{X(k-\mu(k,g)\cdot g,g)} g
\eeq
where $X(\cdot)$ and $\mu(\cdot)$ stand for the corresponding solutions of \eqref{eq:X} and \eqref{eq:mu}.
Indeed, $P_g|_{H(g)}=id$ and for general $k\in Sym^2(M)$ the tensor given by \eqref{def:PH} is well defined and divergence- as well as trace-free with respect to $g$, i.e.~an element of $H(g)$. 
Furthermore, $k-P_g(k)$ stands orthogonal to any $h\in H(g)$ as
\beqas
\langle h,k-P_g(k)\rangle_{L^2(M,g)}&=\langle h,\mu \cdot g\rangle_{L^2}+\langle h,L_Xg\rangle_{L^2}=\int_{M}\mu\cdot tr_g(h)\, dv_g+\langle h,-\delta_g^* X\rangle_{L^2}\\
&=-\langle \delta_g h,X\rangle_{L^2} =0.
\eeqas
To analyse the dependence of $P_g$ on $g$ we now use that $X$ and $\mu$ are characterised by elliptic PDEs for which the following uniform estimates apply
\begin{lemma}
\label{lemma:elliptic} 
Let $s>3$ and $g_0\in \M$ be given and let  $S=S(g_0,s)$ be a sufficiently small slice.
Then there exists a constant $C=C(s,g_0)<\infty$ such that the following claims hold true for every $g\in S$.
For every vector field $Y$ there is a unique solution of the equation 
\beq \label{eq:elliptic}
\delta_g\delta_g^* X=Y
\eeq
and for any $0\leq l\leq s+1$ we have
\beqs
\norm{X}_{H^{l}}\leq C\cdot\norm{Y}_{H^{l-2}}.
\eeqs
Similarly, the unique solution $\mu$ of 
\beq \label{eq:elliptic2}
-\Delta_g\mu +2\mu=f\in H^{l-2}(M,\R),
\eeq
satisfies 
\beqs
\norm{\mu}_{H^{l}}\leq C\cdot\norm{f}_{H^{l-2}}, \quad 0\leq l\leq s.
\eeqs
\end{lemma}
We remark that the occurring Sobolev norms with negative exponent are to be understood as the norms of the coefficients in the dual spaces 
$H^{-k}(\Omega)=(H_0^k(\Omega))^{*}$, $\Omega\subset \R^2$.

The reason why the solution $X$ of \eqref{eq:elliptic} is unique is that we work on a surface that has negative curvature. Thus the kernel of $\de_g\de_g^*$, which agrees 
with the space of Killing-fields, is trivial, see e.g.~\cite{Kob-Nom}, Thm.~5.3. Elliptic regularity theory combined with the Fredholm alternative theorem then immediately gives the estimates
for each individual $g\in S$. These estimates are indeed uniform since all metrics in $S$ are contained in a small ($H^s$) neighbourhood of $g_0$.

We can now give the proof of Proposition \ref{Prop:S}, first for metric contained in the slice. 

Let $g_0\in \M$, $s>3$ and let $S=S(g_0,s)$ be a small slice as defined above. 
Combining the elliptic estimates of Lemma \ref{lemma:elliptic} with \eqref{def:PH} 
and the bounds on $g$ given in Lemma \ref{lemma:S}, we find that for every $0\leq l\leq s$ and every $k\in Sym^2(M)$
\beqa\label{est:PH-proof1}
\norm{P_g(k)}_{H^l}&\leq \norm{k}_{H^l}+C\norm{\mu}_{H^l}+C\norm{X}_{H^{l+1}}\\
&\leq \norm{k}_{H^l}+C(\norm{DR(g)(k)}_{H^{l-2}}+\norm{\delta_g k}_{H^{l-1}})\leq C\norm{k}_{H^l}.
\eeqa
Here and in the following we crucially use that Lemma \ref{lemma:S} gives bounds on $s+1$ derivatives of $g$ so that we may estimate the 
$H^s$ and not just the $H^{s-1}$ norm of Lie-derivatives $L_Xg$. 

Similarly, given any $C^1$ curve $g$ in the slice, we differentiate the corresponding equations \eqref{eq:X} and \eqref{eq:mu} characterising $X$ and $\mu$. This leads to  elliptic 
PDEs of the form \eqref{eq:elliptic} and \eqref{eq:elliptic2} for  
$\ddt X(t)$ and $\ddt\mu(t)$. Applying Lemma \ref{lemma:elliptic} and making use of the bound $\norm{\ddt g}_{H^{s+1}}\leq C\cdot \norm{\ddt g}_{H^s}$ of Lemma
\ref{lemma:S}, we obtain
\beq \label{est:PH-proof2}
\norm{\frac{d}{dt} P_{g(t)}(k)}_{H^{l}}\leq C\norm{\ddt g}_{H^s}\cdot \norm{k}_{H^l}
\eeq
for any (sufficiently smooth) tensor $k\in Sym^2(M)$ and any $0\leq l\leq s$. 

In order to establish the estimates \eqref{est:Pg1} and \eqref{est:Pg2} claimed in Proposition \ref{Prop:S} we now need to prove that 
the two estimates \eqref{est:PH-proof1} and \eqref{est:PH-proof2} obtained above remain valid with the $H^l$ norm on the right hand side replaced with the $L^2$ norm. We use

\textit{Claim:} There exists $C<\infty$ such that for all $g\in S$ and all $h\in H(g)$  
\beqs \norm{h}_{H^{s}}\leq C\norm{h}_{L^2(M,g)}. 
\eeqs

\textit{Proof of Claim:} 
The estimate trivially holds true for $g=g_0$ (or indeed for any one fixed metric) since $H(g_0)$ is a finite dimensional space of smooth tensors. 
For general $g\in S$ we can parametrize $H(g)$ over $H(g_0)$ by restricting the projection $P_g$ onto $H(g_0)$. Using estimate \eqref{est:PH-proof1}, we then get
\beq\label{est:H(g)}\norm{P_g(k)}_{H^{s}}\leq C\norm{k}_{H^{s}}\leq C\norm{k}_{L^2} \text{ for every } k\in H(g_0), g\in S,\eeq
with $\norm{\cdot}_{L^2}$ denoting one of the equivalent $L^2(M,g)$ norms, $g\in S$, say $\norm{\cdot}_{L^2(M,g_0)}$.

On the other hand, integrating \eqref{est:PH-proof2} for $l=0$ along a suitable curve of metrics connecting $g_0$ to $g$ and making use of the fact that $P_{g_0}\vert_{H(g_0)}=id$, we obtain
that for any $k\in H(g_0)$
\beqas
\norm{k}_{L^2}&\leq \norm{P_g(k)-k}_{L^2}+\norm{P_g(k)}_{L^2}\leq C d_S(g,g_0)\cdot \norm{k}_{L^2}+\norm{P_g(k)}_{L^2}\\
&\leq \frac12 \norm{k}_{L^2}+\norm{P_g(k)}_{L^2}\eeqas
provided the slice is chosen small enough. 
Combined with estimate \eqref{est:H(g)} this implies the claim for tensors in the image $P_g(H(g_0))\subset H(g)$ which must agree with $H(g)$ because 
$P_g\vert_{H(g_0)}$ is injective and $\dim(H(g))=\dim(H(g_0))$. 

Combining this claim with the estimate \eqref{est:PH-proof1} for $l=0$ we have thus proved the first claim \eqref{est:Pg1} of Proposition \ref{Prop:S} for general tensors $k\in Sym^2(M)$ and for 
metrics $g\in S$ in the slice. 

To obtain an improved version of \eqref{est:PH-proof2}, we write 
$$P_{g(t)}(k)=P_{g(t)}\big(P_{g(t_0)}(k)\big)+P_{g(t)}\big(P_{g(t)}(k)-P_{g(t_0)}(k)\big)$$
and estimate the derivative of the right hand side at $t=t_0$. Estimate \eqref{est:PH-proof2}, applied first for $l=s$ and then for $l=0$, combined with the estimate \eqref{est:Pg1} we just proved
then implies that for any $k\in Sym^2(M)$
\beqas \norm{(\ddt P_{g(t)}(k))(t_0)}_{H^{s}}&\leq \norm{\ddt g}_{H^s}\cdot \norm{P_{g(t_0)}(k)}_{H^{s}}+\norm{P_{g(t_0)}(\ddt P_{g(t)}(k))}_{H^{s}}\\
&\leq C\norm{\ddt g}_{H^s}\cdot \norm{k}_{L^2}+\norm{\ddt P_{g(t)}(k)}_{L^2}\leq C\norm{\ddt g}_{H^s}\cdot \norm{k}_{L^2}.\eeqas

This completes the proof of Proposition \ref{Prop:S} for metrics $g$ contained in the slice. We now pull back these estimates to the full $H^s$ neighbourhood $W$ of $g_0$ given by the slice-theorem \ref{thm:Tromba}.
The key observation allowing us to do so is that the projection onto the horizontal space commutes with the pull-back
$$f^*P_g(k)=P_{f^*g}(f^*k).$$

Thus, given a $C^1$ curve $g$ in $W$, we write it (uniquely) in the form $g(t)=f(t)^*\bar g(t)$, for $f(t)\in V\subset \D_0^{s+1}$ and $\bar g(t)\in S$ and recall that $\norm{\ddt \bar g}_{H^s}$ and 
$\norm{\ddt f}_{H^{s+1}}$ are controlled by $\norm{\ddt g}_{H^s}$, see Theorem \ref{thm:Tromba}. Indeed, since the diffeomorphisms $f$ are contained in a neighbourhood of the identity, also $\norm{\ddt f^{-1}}_{H^{s+1}}$ is bounded in this way. 
Applying estimates \eqref{est:Pg1} and \eqref{est:Pg2} for $\bar g\in S$, we thus find 
$$\norm{P_g(k)}_{H^s}=\norm{f^*(P_{\bar g}((f^{-1})^*k)}_{H^s}\leq C\cdot \norm{P_{\bar g}((f^{-1})^*k)}_{H^s}\leq C\norm{k}_{L^2}$$
as well as 
\beqas
\norm{\ddt P_{g(t)}(k)}_{H^s}& \leq C\cdot \norm{\ddt f}_{H^{s+1}}\cdot \norm{P_{\bar g}((f^{-1})^*k)}_{H^s}+C\norm{\ddt (P_{\bar g(t)}((f(t)^{-1})^*k)}_{H^s}\\
&\leq C\cdot \big(\norm{\ddt f}_{H^{s+1}}+ \norm{\ddt \bar g}_{H^s}+\norm{\ddt f^{-1}}_{H^{s+1}}\big)\cdot \norm{k}_{L^2}\\
&\leq C\cdot \norm{\ddt g}_{H^s}\cdot \norm{k}_{L^2}
\eeqas
for any tensor $k\in Sym^2(M)$ and any curve in $W$ as claimed in Proposition \ref{Prop:S}.\end{proof}
\textit{Proof of Proposition \ref{Prop:horizontal}.}
For any number $s>3$ we define a function $\th:\M^s\to [0,\infty]$ as follows. 
For any metric $g_0\in\M^s$ we let $\th(g_0)$ be the supremum of all numbers $\th\geq 0$ such that there exists a number $C<\infty$ for 
which estimate \eqref{est:Ck} holds true for all (piecewise) horizontal
curves in $\M^s$ of length $L_{L^2}(g)\leq\th$ and with $g(0)=g_0$. We stress that both this constant $C$, as well as the constant in Proposition \ref{Prop:horizontal}, are allowed to depend on 
the metric $g_0$. 

We first claim that the function $\theta$ is strictly positive for all \textit{smooth} metrics. So let $g_0\in \M$ and let $W$ be the neighbourhood of $g_0$ in $\M^s$ for which Proposition \ref{Prop:S} applies. 
Writing the velocity of any horizontal curve as 
$\ddt g=P_g(\ddt g)$ and applying Proposition \ref{Prop:S} we find that 
$$\norm{\ddt g}_{H^s}\leq C\norm{\ddt g}_{L^2(M,g)}$$
for as long as the curve is contained in $W$. But $W$ is an $H^s$ neighbourhood, so this estimate implies that any curve of small enough $L^2$ length 
and with $g(0)=g_0$ is fully contained in $W$ and thus that indeed $\th(g_0)>0$.

Secondly, we observe that $\th$ is invariant under the pull-back by diffeomorphisms. More precisely let $\D^{s+1}$ be the set of all diffeomorphism of class $H^{s+1}$ (not necessarily homotopic
to the identity). Then we claim that for any $g\in\M^s$ and any $f\in\D^{s+1}$ 
$$\th(f^*g_0)=\th(g_0).$$
Indeed, pulling-back any horizontal curve $g$ in $\M^s$ by a fixed diffeomorphism $f\in \D^{s+1}$ results 
in another horizontal curve of the same $L^2$-length and with velocity bounded by
$\norm{\ddt(f^*g(t))}_{H^s}\leq C\cdot \norm{\ddt g(t)}_{H^s}$, with $C<\infty$ a constant depending on $f$. 
But we defined $\th(g)$ asking only for an estimate of the form \eqref{est:Ck} to be satisfied for \textit{some} constant $C<\infty$, allowed to depend on the considered metric, so the claim follows.

We conclude that $\th$ induces a positive map $\bar\th$ on moduli space $\M/\D$ and now want to prove that this function is continuous with respect to the Weyl-Peterson metric $d_{WP}$.

We recall that the length of a $C^1$ curve $[g]$ in moduli space (with respect to the Weyl-Peterson metric) is given by 
$$L_{WP}([g])=\frac12L_{L^2}(\tilde g)$$
for $\tilde g$ a `horizontal lift'  of 
$[g]$, that is a \textit{horizontal} curve $\tilde g\in \M$ with $[\tilde g(t)]=[g(t)]$ for each $t$. 

Given any two points 
$[g_1]$ and $[g_2]$ in $\M/\D$ we now claim that  
$$\bar\th([g_2])\geq \bar\th([g_1])-2\cdot d_{WP}([g_1],[g_2]),$$ 
and thus switching the roles of $[g_1]$ and $[g_2]$ that $\bar \th$ is Lipschitz continuous on moduli space $(\M/\D,d_{WP})$.
So let $\de>0$ be any fixed number and choose a (piecewise) horizontal path $\tilde g$ of $L^2$-length less than $2\cdot d_{WP}([g_1],[g_2])+\de/2$ 
that connects a representative $f^*g_1$ of $[g_1]$ with $g_2$. Let now $g$ be any given 
(piecewise) horizontal curve with $g(0)=g_2$ and of length $L_{L^2}(g)\leq \bar \th([g_1])-2 d_{WP}([g_1],[g_2])-\de$. Precomposing it with 
$\tilde g$ we obtain a curve $G$ of length $L_{L^2}(G)\leq \bar\th([g_1])-\de/2=\th(f^*g_1)-\de/2$ and with starting point $G(0)=f^*g_1$.
By definition of $\th(f^*g_1)$, the estimate \eqref{est:Ck} is satisfied for the extended curve 
$G$ and thus in particular for $g$ itself, with a constant $C$ depending on $f^*g_1$ and possibly $\de$ but not on $g$. We obtain the claim since $\de> 0$ can be chosen arbitrarily small.  

Given any number $\eps>0$ we now consider the subset $K_\eps$ of moduli space consisting of the equivalence classes of smooth metrics with shortest closed geodesic of length 
no less than $\eps$. This set $K_\eps$ is compact by the Mumford compactness theorem, see e.g.~\cite{Tromba}, p.75. As a positive and continuous function, $\bar \th$ is thus bounded away from zero uniformly on $K_\eps$
which implies Proposition \ref{Prop:horizontal} for smooth metrics.  

For non-smooth metrics $g\in \M^s\setminus \M$, we finally obtain the claim of Proposition \ref{Prop:horizontal} using the invariance of $\th$ under $H^{s+1}$ diffeomorphisms as well as
 
\begin{lemma}\label{lemma:pull-back}
 Given any $g\in\M^s$ there exists a smooth metric $\bar g\in \M$ and a diffeomorphism $f$ of class $H^{s+1}$ such that 
$$g=f^*\bar g.$$
\end{lemma}
For the sake of completeness we provide a proof of this fact in the appendix.

For most arguments in the rest of the paper the estimates of Proposition \ref{Prop:S} and \ref{Prop:horizontal}, controlling the $L^2$-orthogonal projection in terms of the $L^2$ norms of the involved
tensors, would be sufficient, though would in some cases lead to slightly weaker regularity results. 
For the proof of uniqueness of weak solutions carried out
in section \ref{section:uniqueness} it is however crucial that we can extend $P_g$ continuously onto the space of tensors with finite $L^1$ norm

\begin{lemma}\label{lemma:L1}
For any $g_0\in \M$ and any $s>3$ there exists a neighbourhood $W$ of $g_0$ in $\M^s$ such that the following holds true.
The map $P_g$ is Lipschitz-continuous as a map from $W$ to the space of linear maps from $(Sym^2(M),\norm{ }_{L^1})$ to the tangent bundle $T\M^s$, i.e.~there exists a constant $C=C(g_0,s)<\infty$
such that for all $g_1,g_2\in W$ and $k\in Sym^2(M)$
\beq \label{est:L1}\norm{P_{g_1}(k)}_{H^s}\leq C\cdot \norm{k}_{L^1} \text{ and } \norm{P_{g_1}(k)-P_{g_2}(k)}_{H^s}\leq C\cdot d_{\M^s}(g_1,g_2)\cdot \norm{k}_{L^1}.\eeq
\end{lemma}

We remark that there is no need to specify with respect to which metric $g\in W$ the $L^1$ norm is computed as all metrics in $W$ are equivalent. 

We prove these refined estimates on $P_g$ using the following consequence of Proposition \ref{Prop:S}

\begin{lemma}\label{lemma:basis}
For any $g_0\in \M$ and any $s>3$ there exists a neighbourhood $W$ of $g_0$ in $\M^s$ and a constant $C<\infty$ so that we can assign to each metric $g$ in $W$ an $L^2(M,g)$-orthonormal basis 
$\{\Th^j(g)\}_{j=1}^{6\gamma-6}$ of $H(g)$ satisfying 
$$\norm{\Th^j(g_1)-\Th^j(g_2)}_{H^s} \leq C\cdot d_{\M^s}(g_1,g_2), \quad g_{i,2}\in W, \quad j=1\ldots 6\gamma-6=\text{dim}(H(g))$$
\end{lemma}
Lemma \ref{lemma:L1} then immediately follows from $P_g(k)=\sum_j\langle k,\Th^j(g)\rangle_{L^2(M,g)} \Th^j(g).$

\begin{proof}[Proof of Lemma \ref{lemma:basis}]
Let $g_0\in\M^s$ and let $W$ be the neighbourhood of $g_0$ given by Proposition \ref{Prop:S}. We fix any $L^2(M,g_0)$-orthonormal basis $\Th^j(g_0)$, $j=1\ldots 6\gamma-6$, of 
$H(g_0)$ and define 
$$\Th_0^j(g):=P_g(\Th^j(g_0)).$$
According to Proposition \ref{Prop:S} this auxiliary family of tensors depends continuously on $g$, 
$$\norm{\Th_0^j(g_1)-\Th_0^j(g_2)}_{H^s}\leq C\cdot d_{\M^s}(g_1,g_2)$$ so that 
$\{\Th_0^j\}$ is a basis of $H(g)$ provided the neighbourhood $W$ is chosen sufficiently small. Furthermore, as the map asigning to each metric $g$ the inner products 
$$g\mapsto \langle \Th_0^j(g),\Th_0^k(g)\rangle_{L^2(M,g)}$$ is also Lipschitz-continuous on $W$,  
so are the coefficients $a_i^j$ of the orthonormal basis $\Th^j(g)=\sum_{i=1}^j a_i^j(g)\Th_0^i(g)$ of $H(g)$ obtained by Gram-Schmidt orthogonalisation and thus the basis itself. 
\end{proof}

\section{Existence of solutions}\label{section:existence}
In this section we establish the existence of weak solutions to \eqref{1} 
satisfying the properties claimed in Theorem \ref{thm:1}, in particular existing for all times 
unless the metric component degenerates in moduli space. 
As a first step, we prove the following short-time existence result

\begin{lemma}
\label{lemma:short-time}
For any initial metric $g_0\in \M$ and any initial map $u_0\in C^\infty(M,N)$ there exists a smooth solution $(u,g)$ of equation \eqref{1} to initial data $(u(0),g(0))=(u_0,g_0)$
defined on an interval $[0,T)$, $T=T(u_0,g_0)>0$.
\end{lemma}

\begin{proof}[Proof of Lemma \ref{lemma:short-time}]
We first recall that the metric evolves by 
\beq \label{eq:ddtg2} 
\frac{dg}{dt}=\frac{\eta^2}{4}Re(P_g^H(\Phi(u,g)))=\frac{\eta^2}{4}P_g(k(u,g)),\eeq
where $k(u,g)=Re(\Phi(u,g))$, compare \eqref{1b} and \eqref{eq:PgH}.

To simplify notations and without loss of generality, we shall from now on consider the flow with coupling constant $\eta=2$.
We also remark that computing the variation 
\beqs
\dds  E(u,g+sl)\vert_{s=0}=-\frac14 \langle Re(\Phi(u,g)), l\rangle_{L^2} \text{ for all } l\in Sym^2(M)
\eeqs
in local coordinate charts, allows us to write the real part of the Hopf-differential in general (not necessarily conformal) coordinate charts as  
\beqs
k(u,g)=Re(\Phi(u,g))=2u^* G_N-2e(u,g)g,\eeqs
$e(u,g)=\frac12 \abs{\na u}_g^2=\frac12 g^{ij}\partial_{x_i}u\cdot \partial_{x_j}u$ the energy density. 

Using the results of the previous section we can consider equation \eqref{1} as a system consisting of a semilinear parabolic PDE coupled with a differential equation on a Banach manifold $\M^s$ that is defined by a locally Lipschitz continuous vector field.
In such a setting we obtain the existence of a classical solution on a short time interval using a standard iteration argument, which, for the sake of completeness, we outline in the appendix. Given any 
$(u_0,g_0)\in C^{2,\alpha}(M,N)\times \M$ and any number $s>3$ we obtain a solution 
$$(u,g)\in C^{2,1,\alpha}([0,T_s)\times M,N)\times  C^1([0,T),\M^s)$$
of \eqref{1}, defined on a maximal interval $[0,T_s)$. This interval might a priori depend not only on $(u_0,g_0)$ but also on the Banach manifold $\M^s$ on which we solve \eqref{1b}. Indeed, the key step needed to prove that the 
obtained solution $(u,g)$ is actually smooth is to show that this is not the case. So suppose that for some $3<s_1<s_2$ we have $T_{s_1}\neq T_{s_2}$. Since classical solutions of \eqref{1}
are uniquely determined by their initial data, compare section \ref{section:uniqueness}, we remark that the two solutions obtained for the different values of $s$ agree 
for as long as they both exist, that is until time $T_{s_2}<T_{s_1}$. Since the metric component is continuous (as a map into $\M^{s_1}$) up to time $T_{s_1}$ 
there exists a number $\eps>0$ such that the length $\ell(g(t))$ of the shortest closed geodesic of $(M,g(t))$ is no less than $\eps$ on the smaller interval $[0,T_{s_2}]$. 
Using the $H^s$ estimates of Proposition \ref{Prop:horizontal} this allows us to conclude that $g$ is $C^1$ as a curve into 
$\M^{s_2}$ on the closed interval $[0,T_{s_2}]$, compare with the proof of Lemma \ref{lemma:char-sing} below. 

Using Lemma \ref{lemma:pull-back}, we then write $g(T_{s_2})\in\M^{s_2}$ in the form $g(T_{s_2})=f^*\bar g(T_{s_2})$ for an $H^{s_2+1}$ diffeomorphism $f$ and a smooth metric $\bar g\in \M$.
Restarting the flow with the pulled-back initial data $(\bar u(T_{s_2}),\bar g(T_{s_2}))=(u(T_{s_2})\circ f^{-1},(f^{-1})^*g(T_{s_2}))\in C^{2,\alpha}\times \M$ we obtain a solution $(\bar u,\bar g)$ of \eqref{1} in 
$C^{2,1,\alpha}(M\times I)\times C^1(I,\M^{s_2})$
on a time interval $I=[T_{s_2},T_{s_2}+\de)$. But equation \eqref{1} is invariant under the pull-back by diffeomorphisms applied simultaneously to both the map and the 
metric component and solutions of \eqref{1} are unique. Thus the pull-back of $(\bar u,\bar g)$ by $f$ is nothing else than our original solution 
$(u,g)$ so that $g$ is in $C^1([0,T_{s_2}+\de),\M^{s_2})$, leading to a contradiction. 

At this point we are now in a position to argue by a standard bootstrapping argument, using parabolic regularity theory to improve the regularity of $u$, as well as the explicit formula for $P_g$ given in \eqref{def:PH} to analyse higher
order time derivatives of $g$. We obtain that $(u,g)$ is indeed smooth.
\end{proof}
We remark that the results of section \ref{section:horizontal} allow us not only to establish short-time existence of solutions to \eqref{1} but already give the following characterisation of the behaviour
of the metric component at a singular time

\begin{lemma}\label{lemma:char-sing}
Let $(u,g)$ be a smooth solution of \eqref{1} defined (and smooth) on a maximal interval $[0,T_1)$. Then one of the following three statements holds 
\begin{enumerate}
\item[(i)] $T_1=\infty$, or
\item[(ii)] $T_1<\infty$ but as $t\nearrow T_1$ the metrics $g(t)$ converge smoothly to a limit $g(T_1)\in \M$; indeed $g$ can be extended to a Lipschitz continuous curve from the closed interval
$[0,T_1]$ into each of the Banach manifold $\M^s$, $s>3$, or 
\item[(iii)] the metrics degenerate in moduli space at a finite time $T_1$, i.e. $\lim_{t\nearrow T_1} \ell(g(t))=0$.
\end{enumerate}

\end{lemma}

\begin{proof}[Proof of Lemma \ref{lemma:char-sing}]
Assume that $T_1<\infty$ and that the length of the shortest closed geodesics in $(M,g(t))$ does not converge to zero
$$\limsup_{t\nearrow T_1}\ell(g(t))>\eps>0.$$
Then given any number $s>3$ we let $\th=\th(s,\eps)>0$ be the constant of Proposition \ref{Prop:horizontal}. We recall that according to the energy identity \eqref{energy-identity}
the $L^2$-length of the curve $g$ is finite on intervals of finite length. We may thus choose $t_0<T_1$ with $\ell(g(t_0))\geq \eps$ and close enough to $T_1$ such that 
$L_{L^2}(g\vert_{[t_0,T_1)})<\th$. 
Proposition \ref{Prop:horizontal} then implies that $g(t)$ is a Cauchy sequence in $\M^s$ and thus converges to a limit $g(T_1)$ in $\M^s$ as $t\nearrow T_1$. 
Indeed, combining Proposition \ref{Prop:horizontal} with the 
energy identity \eqref{energy-identity} gives $C^{1/2}$-H\"older estimates in time for $g$ considered as map into $\M^s$. 
Moreover, thanks to the 
uniform bound on the energy of $u$ and thus on the $L^1$ norm of the Hopf-differential
$$\norm{k(u,g)}_{L^1}\leq C\cdot\norm{\na u}_{L^2}^2\leq C\cdot E(u,g)\leq C \cdot E(u_0,g_0),$$
the improved estimates on $P_g$ stated in Lemma \ref{lemma:L1} give uniform bounds on  
$\norm{\ddt g(t)}_{H^s}$. Thus $g$ is not only $C^{1/2}$ but indeed Lipschitz continuous with respect to each $H^s$ metric on the \textit{closed} interval 
$[0,T_1]$.
\end{proof}

We remark that the possibility of solutions degenerating in moduli space will be addressed in future work and that here we focus on the analysis of singularities of the second type,
essentially due to the map component becoming singular. 

So let $(u,g)$ be a smooth solution of \eqref{1} on a maximal interval $[0,T)$. Assume that the metrics do not degenerate in moduli space as we approach the singular time 
and thus that $g(t)\to g(T_1)\in\M$ smoothly as $t\nearrow T_1$. We remark that the evolution of the metric component is uniformly controlled, 
\beq \label{est:ddg} \norm{\ddt g}_{H^s}\leq C\norm{k(u,g)}_{L^1}\leq C\cdot \norm{\na u}_{L^2}^2\leq  C\cdot E_0\eeq
for times in an interval of length $\de=\de(g(T_1),s)>0$ not just for the 
one solution $(u,g)$ of \eqref{1} that becomes singular, but also for all solutions evolving from an initial data $(\bar u,g(T_1))$ with energy bounded by $E_0$. 
Thanks to this strong bound on the metric component we can carry out the analysis of the map component of solution to \eqref{1} near singular times using 
methods familiar from the work of Struwe \cite{Struwe85} on the harmonic map flow. 
Since our analysis closely follows the ideas of \cite{Struwe85} we shall omit some details and calculations in the following presentation. We also remark that a similar argument was briefly 
outlined in \cite{Ding-Li-Liu} in the special case of maps from a torus.

Notation: We let $g_1\in\M$ be a fixed metric that should be thought of as a limiting metric of a solution of \eqref{1} at a singular time. 
Then unless indicated otherwise all occurring objects such as norms, operators (like $\Delta$), integrals, balls and so on are to be understood as the corresponding objects on 
the fixed Riemannian surface $(M,g_1)$. 
Furthermore, we denote generic constants (allowed to change from line to line) by $C$ in case they depend only on $g_1$ and $E_0$ and will indicate any dependence on additional 
quantities accordingly. 

Based on \eqref{est:ddg} we henceforth restrict our attention to solutions of \eqref{1} satisfying 
\beq
\label{ass:close}
\norm{g_1-g(t)}_{H^s}\leq \eps_1 
\eeq
for some fixed number $s>3$ and a small $\eps_1=\eps_1(g_1,s)>0$, chosen in particular such that Lemma \ref{lemma:L1} applies on this $\M^s$ neighbourhood of $g_1$.

We first remark that the evolution of the local energy is controlled by
\begin{lemma}\label{lemma:loc-energy} 
For solutions $(u,g)$ of \eqref{1} satisfying \eqref{ass:close} the following local energy bounds hold true for any point
$x\in M$ and any radius $0<r<r_{inj}$ 
\beqs 
E(u(t),B_{r/2}(x))\leq 2 E(u(0),B_r(x))+C\frac{t}{r^2} \eeqs
and
\beqs E(u(t),B_{r}(x))\geq \frac12 E(u(0),B_{r/2}(x))-4 \int_{0}^t\int_M \varphi^2\abs{\pt u}^2 dv dt-C\frac{t}{r^2}. 
\eeqs 
\end{lemma}
\begin{proof}[Sketch of proof]
Given $x\in (M,g_1)$ and $0<r<r_{inj}(M,g_1)$ we let $\varphi\in C_0^\infty(B_r(x),[0,1])$ be a standard cut-off function, i.e.~such that 
$\varphi\equiv 1$ on $B_{r/2}(x)$ and $\abs{\na \varphi}\leq \frac{C}{r}$.
A short calculation shows that for a solution $(u,g)$ of \eqref{1} 
\beqa\label{est:energy}
0&= \int \varphi^2\abs{\pt u}^2 \,dv-\int\varphi^2\pt u\cdot \Delta_{g(t)} u\,dv\\
&=\int  \varphi^2\abs{\pt u}^2 \,dv+\frac12\ddt\int\varphi^2\abs{\na u}_{g(t)}^2\,dv+R(u(t),g(t))
\eeqa
with an error term that is bounded by
$$\abs{R(u,g)}\leq (\frac{C}{r^2}+C\norm{\ddt g}_{C^0})\cdot E(u,B_{r}(x))+\frac18\int_M\varphi^2\abs{\pt u}^2 dv.$$
Since $\norm{\ddt g}_{C^0}$ is uniformly bounded, this estimate integrates to give an upper and a lower bound on 
$\int\varphi^2\abs{\na u}_{g(t)}^2-\int\varphi^2\abs{\na u(0)}_{g(0)}$. Combined with the fact that
$\frac12 g\leq \tilde g\leq 2g$ for all $g,\tilde g$ satisfying \eqref{ass:close}, we obtain the claims of Lemma \ref{lemma:loc-energy}. 
\end{proof}
An important consequence of the previous calculation is
\begin{cor}
Suppose $(u,g)$ is a smooth solution of \eqref{1} defined on a maximal interval $[0,T_1)$ for which \eqref{ass:close} is satisfied. Then for any $\eps_0>0$
the set of points 
\beqs S:=\{x\in M: \limsup_{t\nearrow T_1} E(u(t),B_R(x))\geq \eps_0 \text{ for all }\, R>0\},
\eeqs
is finite.  
\end{cor}
In fact $\#S\leq E_0/\eps_0$, since energy concentrates near points of $S$ not just along a suitable sequence $t_j\nearrow T_1$ but indeed for \emph{all} sequences $t\nearrow T_1$, compare \eqref{est:energy}.

Away from the finite set $S$ we control the map component of the flow using the following lemma which should be seen as the analogue of Lemmas 3.10 and 3.10' of \cite{Struwe85}  
\begin{lemma}\label{lemma:Struwe}
There exists a number $\eps_0>0$ depending only on $g_1$ and $E_0$ such that the following statement holds true. Let $(u,g)$ be a smooth solution of \eqref{1} on an open interval 
$(0,T)$ and assume that \eqref{ass:close} is satisfied.
Let $M'\subseteq M$ be an open set such that there exists a number $R>0$ with
\beq \label{ass:energy} 
E(u(t),B_R(x))\leq \eps_0 \quad\text{ for all } (x,t)\in M'\times (0,T).\eeq
Then the parabolic H\"older-norms of $u$ and its spatial derivatives (upto order $s-2$) are bounded uniformly on the sets $[\tau,T]\times M'$, $\tau>0$, with bounds depending only
on $\tau$, $R$, $g_1$, $T$, $s$, $M'$ and the energy bound $E_0$.
\end{lemma}
\begin{rmk}\label{remark:Struwe}
If the initial map $u(0)$ is smooth on a neighbourhood of $M'$ then 
the above result can be extended to give bounds on the H\"older norms of $u\vert_{M'}$ and its spatial derivatives on $M'$ up to time $t=0$, now with bounds depending additionally on $u(0)$, compare with remark 
3.11 and 3.11' of \cite{Struwe85}.
\end{rmk}

\begin{rmk}\label{remark:higher-reg}
Because of the non-local nature of the projection operator $P_g$ these estimates on $u\vert_{M'}$ allow us to improve the regularity of $g\vert_{M'}$ from the a priori known $C^{0,1}$ 
dependence on time only in case $M= M'$. 
For $M'\neq M$ we can improve the bounds of Lemma \ref{lemma:Struwe} to give $C^{1,\alpha}$ bounds in time on $u\vert_{M'}$ and its spatial derivatives while for $M=M'$, i.e.~away from singluar 
\textit{times}, a bootstrapping argument gives estimates on any $C^k$ norm (in space and time) of $(u,g)$ in terms of the quantities specified in Lemma \ref{lemma:Struwe}.
\end{rmk}

\begin{proof}[Proof of Lemma \ref{lemma:Struwe}]
For the proof of this lemma we follow largely the ideas of \cite{Struwe85}. We make use of the well known interpolation estimate, see e.g.~\cite{Ding-Tian}
\begin{lemma} \label{lemma:H2}
There are numbers $\eps_0>0$ and $C<\infty$ (depending on $(M,g_1)$ and the target manifold) such that for all maps $u\in H^2(M,N)$, a bound on the local energy of 
\beqs
E(u,B_{r}(x))\leq \eps_0\eeqs
implies an $H^2$-bound of the form
\beq \label{est:H2}
\int \varphi^2\abs{\na^2 u}^2 dv\leq \frac{C}{r^2} E(u,B_{r}(x))+C\int \varphi^2 \abs{\tau(u)}^2dv,\eeq
as well as an estimate of
\beq \label{est:W14}
\int \varphi^2\abs{\na u}^4 dv\leq C E(u,B_{r}(x))\cdot \bigg[\frac{1}{r^2} E(u,B_{r}(x))+\int \varphi^2 \abs{\tau(u)}^2dv\bigg].\eeq
Here $\varphi\in C^\infty_0(B_r(x))$ denotes a cut-off function.
\end{lemma}
Let now $(u,g)$, $M'$ and $R>0$ be as in Lemma \ref{lemma:Struwe}, let $x\in M'$ and choose a cut-off function $\varphi\in C^\infty_0(B_{R/2}(x))$.
We first remark that for $\eps_1>0$ sufficiently small, the pointwise bound
$\abs{\tau_g(u)-\tau(u)}\leq C\eps_1 (\abs{\na^2 u}+\abs{\na u}^2)$ implies that 
\eqref{est:H2} and \eqref{est:W14} remain valid with $\tau(u(t))$ replaced by $\tau_{g(t)}(u(t))=\pt u(t)$. 

As in \cite{Struwe85} we now differentiate equation \eqref{1a} in time and multiply with $\varphi^2\pt u$. After carefully analysing all occurring terms, in 
particular the terms due to the time-dependence of the metric, we 
find 
\beqas 
\frac12\ddt \int \varphi^2\abs{\pt u}^2 +\int\varphi^ 2 \abs{\na\pt u}^2  &
\leq (\frac14+C\eps_1)\cdot \int\varphi^ 2 \abs{\na\pt u}^2 + C\int\varphi^2 \abs{\pt u}^2\abs{\na u}^2\\
&\quad +C(R)\big(1+\norm{\frac{dg}{dt}}_{C^2}^2+\int \abs{\pt u}^2\big).
\eeqas
Since we know that the $H^s$ norm, and thus also the $C^2$ norm, of $\ddt g$ is uniformly bounded by a multiple of the energy, we obtain that for $\eps_1=\eps_1(g_1)>0$ chosen small enough
\beq \label{est:ddtdtu} \ddt \int \varphi^2\abs{\pt u}^2 +\int\varphi^ 2 \abs{\na\pt u}^2  \leq C(R)(1+\frac{dE}{dt})+C\int\varphi^2 \abs{\pt u}^2\abs{\na u}^2.\eeq

Using the Sobolev embedding $W^{1,1}\hookrightarrow L^2$ as well as Lemma \ref{lemma:H2}, we find 
\beqas
C\int\varphi^2 \abs{\pt u}^2\abs{\na u}^2&\leq C\norm{\varphi \abs{\pt u}^2}_{L^2}\cdot \big(\int\varphi^2\abs{\na u}^4\big)^{1/2}\\
&\leq C(\norm{\na(\varphi \abs{\pt u}^2)}_{L^1}+\norm{\abs{\pt u}^2}_{L^1})\eps_0^{1/2} \big[C(R)+ \int\varphi^2\abs{\pt u}^2\big]^{1/2}\\
&\leq \int\varphi^ 2 \abs{\na\pt u}^2+C\eps_0\frac{dE}{dt}\int\varphi^2\abs{\pt u}^2+C(R)\cdot (\frac{dE}{dt}+1).\eeqas 
Integrating the resulting estimate \eqref{est:ddtdtu} over any interval $[t_1,t_2]\subset (0,T)$ we thus obtain
$$
\int \varphi^2\abs{\pt u}^2 dv\Big\vert_{t_1}^{t_2}\leq C \eps_0\cdot\sup_{t\in [t_1,t_2]}\int \varphi^2\abs{\pt u}^2 dv+C(R,T).$$

After possibly reducing $\eps_0=\eps_0(E_0,g_1)>0$ so that the factor $C\eps_0\leq \frac12$, we conclude that for any $\tau>0$
\beqas
\sup_{t\in[\tau, T)}\int_{B_{R/4}(x)}\abs{\pt u}^2 dv
& \leq 2\inf_{t\in [0,\tau)}\int \varphi^2\abs{\pt u(t)}^2 dv+ C(R,T)\\
&\leq 2\frac{E_0}{\tau}+C(R,T).\eeqas
Repeating the above argument for a finite cover of balls $B_{R/4}(x_i)$ of $M'$ we obtain a uniform estimate of
$$\int_U\abs{\pt u(t)}^2 dv\leq C \quad \text{ for all } t\in [\tau,T)$$
on a small neighbourhood $U$ of $M'$. According to Lemma \ref{lemma:H2} this implies a bound on $\int_{U'}\abs{\na^2 u(t)}^2 dv$ on a slightly smaller neighbourhood of $M'$. Applying Sobolev's embedding theorem we then conclude 
that for any exponent $p<\infty$ 
$$\int_{U'}\abs{\na u(t)}^p dv \leq C_p\qquad t\in [\tau,T).$$ 
Then as in \cite{Struwe85} we think of \eqref{1a} as an inhomogeneous heat equation 
$$\pt u-\Delta_{g} u=A_g(u)(\na u,\na u)\in L^p(U'\times[\tau,T))$$
allowing us to apply standard regularity results for parabolic equations, see e.g.~\cite{Lieberman}, chapter VII; we get bounds in the parabolic Sobolev-spaces 
$W_p^{2,1}$, and thus in the parabolic H\"older spaces $C^{\alpha}$, on sets $M'\times [\tau',T]$, $\tau'>\tau$. We finally obtain estimates on the H\"older norms of spatial derivatives of $u$ (up to order $s-2$) by a standard bootstrapping argument
which relies on the strong bounds on the velocity of horizontal curves given in Lemma \ref{lemma:L1}.\end{proof}

Let now $(u,g)$ be a smooth solution of \eqref{1} on $[0,T_1)$ whose metric component does not degenerate and thus smoothly converges $g(t)\to g(T_1)=:g_1\in\M$ as $t\nearrow T_1$. 
We first remark that the uniform bounds on the energies $E(u(t))\leq 2E(u(t),g(t))\leq 2E_0$ 
combined with the fact that $\pt u\in L^2([0,T_1)\times M)$ imply that the maps $u(t)$ converges weakly in 
$H^1(M)$ to a limit $u(T_1)$ as $t\nearrow T_1$. Additionally, Lemma \ref{lemma:Struwe} gives uniform H\"older bounds on $u$ and its \textit{spatial} derivatives away from the finite set
$S$ of concentration points so that $u(t)$ converges also in $C^\infty_{loc}(M\setminus S)$.

We now remark that any concentration of energy must be due to the so-called bubbling off of (at least) one harmonic sphere. 
Indeed, the analysis carried out in \cite{Struwe85} (p. 578/9) remains unchanged as long as the local energy estimates and $H^2$ bounds used in \cite{Struwe85} are replaced by Lemmas
\ref{lemma:loc-energy} and \ref{lemma:H2}. We obtain the following: For any point $x_0\in S$ there are sequences of times $t_i\nearrow T_1$, radii $r_i\to 0$ and points $x_i\to x_0$ with energies on balls around $x_i$ of 
$$E(u(t_i),B_{2r_i}(x_i))\leq \eps_0 \quad \text{ and }\quad E(u(t_i),B_{r_i}(x_i))\geq c\eps_0, \quad c=c(M,g_1)>0$$
and with tension satisfying
$$r_i^2\int_{B_{2r_i}(x_i)}\abs{\tau(u(t_i))}^2 \to 0 \quad \text{ as } i\to \infty.$$
The rescaled maps $u_i(x)=u(\exp_{x_i}(r_ix),t_i)$, defined on larger and larger subsets of $\R^2$ are then bounded uniformly in $H^2$ and subconverge (weakly in $H^2$, 
strongly in $W^{1,p}$, $p<\infty$) to a non-constant harmonic map of finite energy that is defined on $\R^2\cup\{\infty\}\cong S^2$, called a harmonic sphere or bubble. The amount of energy that 
concentrates near $x_0$, and that is consequently lost as we pass to the limit $t\nearrow T_1$, is no less than $\eps^*=\eps^*(N)$, the minimal energy of such a non-constant harmonic map from 
$S^2$ to the target. 

Finally, as in \cite{Struwe85}, we (weakly) continue the flow past any such singular time by restarting from initial data $g(T_1)\in \M$ and $u(T_1)\in H^1(M,N)\cap C_{loc}^\infty(M\setminus S)$ as follows: let 
$u_{j,0}\in C^{\infty}(M,N)$ be a sequence of maps 
that converge to $u(T_1)$ in $H^1(M)$ as well as in $C^\infty_{loc}(M\setminus S)$ and let $(u_j,g_j)$ be the smooth solution of \eqref{1}  corresponding to initial data $(u_{j,0},g(T_1))$ that
exists at least on some interval $[T_1,T_1+\de_j)$ according to Lemma \ref{lemma:short-time}. We remark that the metrics $g_j$ are uniformly Lipschitz continuous and thus that the estimates derived above 
can be applied on $[T_1,T_1+\min(\de_j,\de_0))$ for a number $\de_0(g(T_1))>0$ independent of $j$.

We now choose $r>0$ such that 
$\sup_{x\in M} E(u_{j,0},B_r(x))<\eps_0/4$, which is possible due to the strong $H^1$-convergence of the initial maps. Then the local energy estimates of Lemma 
\ref{lemma:loc-energy} imply
that there is no concentration of energy and thus in particular no blow-up for any of the maps $u_j$, on a uniform interval $I=[T_1,T_1+cr^2]$, $c=c(g_1)>0$, in the sense that    
$$E(u_j(t),B_{r/2}(x))<\eps_0 \text{ for all } j\in \N,\, x\in M,\, t\in I.$$
According to Lemma \ref{lemma:Struwe} as well as remarks \ref{remark:Struwe} and \ref{remark:higher-reg}
we thus obtain uniform $C^{1,\alpha }$ estimates in time for the maps $u_j$ (and their spatial derivatives)
in every compact 
subset of $M\times[T_1,T_1+cr^2]\setminus (S\times\{T_1\})$. Away from the singular time, we furthermore get uniform bounds on all $C^k$ norms of $(u_j,g_j)$ in space-time. We conclude that 
a subsequence of $(u_j,g_j)$ converges smoothly on $M\times (T_1,T_1+cr^2]$ to a pair $(u,g)$ 
which solves \eqref{1} classically on 
$(T_1,T_1+cr^2]$ and weakly on $[T_1,T_1+cr^2]$. This solution achieves the initial data $(u(T_1),g(T_1))$ in the sense 
that for $t\searrow T_1$ the maps $u(t)$ converges to $u(T_1)$
weakly in $H^1(M)$ and 
smoothly away from the set $S$ while the metric component $g$ is Lipschitz-continuous across the singular time. 
Since the energy of the approximating solutions $(u_j,g_j)$ is no more than $E(u_{j,0},g(T_1))\to E(u(T_1),g(T_1))$, the extended weak solution $(u,g)$ has non-increasing energy also across the singular
time $T_1$. In particular the total number of all singular points $\cup_iS(T_i)\times\{T_i\}$ of such a solution is bounded by $\frac{E(u(0),g(0))}{\eps^*}$. After possibly repeating 
the above argument to analyse any further singularities, we thus obtain a weak solution satisfying the properties (i)-(iii) of Theorem \ref{thm:1} and existing for 
as long as the metrics do not degenerate in moduli space. 

\section{Uniqueness of weak solutions}\label{section:uniqueness}
We finally discuss the issue of uniqueness of weak solutions. We prove that the solution $(u,g)$ of \eqref{1} constructed in the previous section is uniquely determined by its initial data,  
not only among all solutions satisfying the properties of Theorem 
\ref{thm:1}, but in the natural class of all weak solutions with non-increasing energy. 
We remark that a further argument as carried out in \cite{Rupflin} actually gives uniqueness under the weaker assumption
that the total energy does not instantaneously increase by more than a certain quantum at any time. We also remark that it is necessary to impose restrictions
on the evolution of the total energy in view of the possibility of reverse bubbling, see \cite{Topping}. 

So let $(u_i,g_i)_{i=1,2}$ be two weak solutions of \eqref{1} defined on an interval $[0,T)$ that evolve from the same intial data 
$$(u_1,g_1)(0)=(u_0,g_0)=(u_2,g_2)(0)\in H^1(M,N)\times \M$$
and assume that the total energies $t\mapsto E(u_{1,2}(t),g_{1,2}(t))$ are non-increasing. 
Since
$$I:=\{t\in[0,T):\, (u_1,g_1)\equiv (u_2,g_2) \text{ on } [0,t]\}$$
is trivially closed in $[0,T)$, we need to prove that $I$ is also open. 

Given any $t_0\in I$ we recall that $g_i(t)\to g_i(t_0)$ in each $\M^s$ and thus certainly uniformly as $t\searrow t_0$. 
Combined with the fact that $u_i(t)\to u_i(t_0)$ strongly in $L^2$ and weakly in $H^1$, we thus obtain 
$$E(u_i(t_0),g_i(t_0))\geq \lim_{t\searrow t_0} E(u_i(t),g_i(t))=\lim_{t\searrow t_0} E(u_i(t),g_i(t_0))\geq E(u_i(t_0),g_i(t_0)),$$ 
where we used the assumption on the evolution of the energy in the first step. Thus $u_i(t)\to u_i(t_0)$ indeed strongly in $H^1(M,g_0)$, which implies in particular that local energies, say on balls,
converge as $t\searrow t_0$. 
Choosing a finite cover of balls $B_r(x_i)$, $i=1\ldots K$, of $(M,g_i(t_0))$ such that 
$E(u_i(t_0),g_i(t_0),B_{2r}(x_i))\leq \eps_0/2$, we may thus choose  
$\de>0$ so small that
$$E(u_i(t),g_i(t_0),B_{2r}(x_i))\leq \eps_0 \text{ for } t\in [t_0,t_0+\de], \, i=1,2.$$
Here we let $\eps_0>0$ be the constant of Lemma \ref{lemma:H2}.

It is now crucial to remark that on almost every time slice the functions 
$u_i(t)$ weakly solves an \textit{almost harmonic map equation}, that is an equation of the form $\tau_{g} v = f$ for a function $f\in L^2$ and a metric $g$, here of course $g=g_i(t)$ and $f=\pt u_i(t)$.
Since any weak solution of such an elliptic equation is contained in the Sobolev space $H^2$, see e.g. \cite{Rupflin}, Proposition 2.1, we may apply Lemma \ref{lemma:H2} 
on almost every time slice resulting in an estimate of   
$$\int_M\abs{\na u_i(t)}^4 +\abs{\na^2 u_i(t)}^2 dv_{g_0}\leq C(r)\cdot (1+\int_M\abs{\pt u(t)}^2 dv_{g(t)}) $$ for $t\in [t_0,t_0+\de)$ and $i=1,2$.
We can thus reduce the uniqueness statement in the general class of weak solutions with non-increasing energy 
to the following lemma whose 
analogue for the harmonic 
map flow was proven in \cite{Struwe85}

\begin{lemma} \label{lemma:unique}
Let $(u_1,g_1)$ and $(u_2,g_2)$ be weak solutions of \eqref{1} to the same initial data $(u_1,g_1)(0)=(u_2,g_2)(0)$ and suppose that  
\beq \label{ass:L4} \na u_i\in L^4(M\times[0,T)),\text{ and } \na^2u_i\in L^2(M\times [0,T)) \quad i=1,2.
\eeq
Then $(u_1,g_1)\equiv (u_2,g_2)$. 
\end{lemma}

\begin{proof}[Proof of Lemma \ref{lemma:unique}]
Using an open-closed argument as above it is enough to prove that the solutions agree on a possibly smaller interval $[0,\de)$, which we can chose in particular such that the metrics 
$g_{1,2}$ are contained in an $H^s$ neighbourhood of $g_0=g_1(0)=g_2(0)$ for which Lemma \ref{lemma:L1} applies. Here $s$ can be chosen to be any fixed number $s>3$. 

Notation: For the following computations we denote by $\norm{\cdot}_{L^p}$ the $L^p(M,g_0)$ norm and by 
$d(\cdot,\cdot)$ the metric on $\M^s$ respectively by $\norm{\cdot}$ the $H^s$ norm on $T\M^s$. Furthermore, we use the short-hand notation of 
$\abs{\na V}:=\max\{\abs{\na u_1},\abs{\na u_2}\}$ which, by assumption, is a function in $L^4(M\times [0,T])$ with $L^2$ norm on time-slices bounded by the energy, 
$\norm{\na V(t)}^2_{L^2(M)}\leq C\cdot(E(u_1,g_1)+E(u_2,g_2))\leq C E(u_0,g_0)$.

Subtracting the equations \eqref{1a} for the map components 
$u_i$ we obtain that the difference $w=u_1-u_2$ satisfies
\beq 
\pt w-\Delta_{g_1}w=(\Delta_{g_1}-\Delta_{g_2})(u_2) +A_{g_1}(u_1)(\na u_1,\na u_1)-A_{g_2}(u_2)(\na u_2,\na u_2)\\
 \label{eq:uniqueness}
\eeq
where $A$ denotes the second fundamental form of the target $N\hookrightarrow \R^N$, 
$A_g(u)(\na u,\na u):=g^{ij}A(u)(\partial_i u, \partial_j u)$.

Following \cite{Struwe85} we multiply equation \eqref{eq:uniqueness} with $w$, integrate over the fixed surface $(M,g_0)$ and estimate the resulting terms using H\"older's inequality. This leads to 
\beqa
\frac12\ddt \norm{w}_{L^2}^2+\norm{\na w}_{L^2}^2\leq&\, C\cdot d(g_1,g_0)\cdot \norm{\na w}_{L^2}^ 2\\
&+C\cdot d(g_1,g_2)\cdot \big(\norm{\na w}_{L^2} +\norm{w}_{L^2} \cdot (1+\norm{\na V}^2_{L^4})\big)\\
&+ C\norm{\na w}_{L^2}\cdot \norm{\na V}_{L^4}\cdot\norm{w}_{L^4}+C\norm{\na V}_{L^4}^2\cdot \norm{w}_{L^4}^2\\
\leq& \,(\frac14 +C\cdot t)\norm{\na w}_{L^2}^2+ C\cdot d(g_1,g_2)^2\\
&+C(1+\norm{\na V}^2_{L^4})\cdot \norm{w}_{L^4}^2.
\label{est:uniqueness} \eeqa

Using the Sobolev embedding $W^{1,1}\hookrightarrow L^2$, we may furthermore estimate  
$$ \norm{w}_{L^4}^2=\norm{w^2}_{L^2}\leq C(\norm{\na(w^2)}_{L^1}+\norm{w^2}_{L^1})\leq C\cdot \norm{w}_{L^2}\cdot(\norm{w}_{L^2}+\norm{\na w}_{L^2})$$
so that the last term on the right hand side of \eqref{est:uniqueness} is bounded by 
$$C(1+\norm{\na V}^2_{L^4})\cdot \norm{w}_{L^4}^2\leq \frac18 \norm{\na w}_{L^2}^2+ C\psi(t)\cdot \norm{w}_{L^2}^2,$$
for $\psi(t)=(\norm{\na V(t)}_{L^4}^4+1)\in L^1([0,T])$.

In order to estimate the distance of the metric components $g_1$ and $g_2$ in terms of $w$ we recall that the evolution of the tensor $g_1-g_2$ is given by 
$$\ddt(g_1-g_2)=P_{g_1}(k(u_1,g_1))-P_{g_2}(k(u_2,g_2)),$$
$k(u,g)=2u^* G_N-2e(u,g)g$. We have a pointwise estimate of the difference of the involved tensors of 
$$\abs{k(u_1,g_1)-k(u_2,g_2)}\leq C\cdot d(g_1,g_2)\cdot \abs{\na V}^2+C\cdot \abs{w}\cdot \abs{\na V}^2+C\cdot \abs{\na w}\cdot \abs{\na V}.$$
Remark, that any $L^2$ estimate of this tensor would involve integrals of the form $\int\abs{\na V}^4 \abs{w}^2$ and $\int\abs{\na V}^2 \cdot \abs{\na w}^2$ which are 
\textit{not} controlled by the quantities of the left hand side of \eqref{est:uniqueness}. It thus crucial at this point that the 
improved bounds on $P_g$ given in Lemma \ref{lemma:L1} only ask for $L^1$ bounds on the involved tensors, allowing us to estimate 
\beqa\label{est:g1} 
\ddt d(g_1,g_2)& \leq C\cdot d(g_1,g_2)\cdot \norm{k(u_1,g_1)}_{L^1}+
C\cdot \norm{k(u_1,g_1)-k(u_2,g_2)}_{L^1}\\
&\leq C\cdot d(g_1,g_2) +C\cdot \norm{\na w}_{L^2}+C\cdot \psi(t)^{1/2}\norm{w}_{L^2}.\eeqa
Gronvall's lemma thus leads to an estimate of 
\beqa d(g_1,g_2)(t)^2&\leq C\cdot \bigg(\int_0^t\norm{\na w(s)}_{L^2(M)}\, ds\bigg)^2+C\bigg(\int_0^t\psi(s)^{1/2}\cdot\norm{w(s)}_{L^2(M)\,} ds \bigg)^2\\
& \leq t\cdot \int_0^t\int_M\abs{\na w}^2 +C\int_0^t\psi(s) ds \cdot \int_0^t\int_M w^2,\eeqa
which we insert into \eqref{est:uniqueness}. Integrating the resulting estimate over time, we find
$$\norm{w(t)}_{L^2}^2+\int_0^t\int\abs{\na w}^2\leq C\cdot \int_0^t\psi(s)ds \cdot \sup_{s\in[0,t]}\norm{w(s)}_{L^2}^2+Ct^2\int_0^t\int\abs{\na w}^2.$$
Since $\psi$ is integrable, we conclude that for all $t$ sufficiently small, say $t\in(0,t_0)$,
$$\int\abs{w(t)}^2\leq \frac12\sup_{s\in[0,t]}\int\abs{w(s)}^2.$$
Thus $w$ must vanish identically on $(0,t_0)$ so $u_1\equiv u_2$ and $g_1\equiv g_2$ as desired. 
\end{proof}
\renewcommand{\thesection}{\Alph{section}}
\setcounter{section}{0}
\numberwithin{equation}{section}
\section{Appendix}
\subsection{Solving the equation on a fixed Banach manifold} $ $\\
Let $(u_0,g_0)\in C^{2,\alpha}(M)\times \M$, $\alpha>0$ be given and let $s>3$ be a fixed number. Here we outline an iteration argument 
that can be used to obtain a solution $(u,g)\in C^{2,1,\alpha}([0,\de)\times M)\times C^1([0,\de), \M^s)$ of \eqref{1} for such initial data.

For $\de_0=\de_0(u_0,g_0,s)>0$ to be determined later, we extend $u_0$ to a constant in time map defined on $M\times [0,\de_0^2)$ and 
define iteratively for $i=1\ldots$
\begin{itemize}
 \item $g_i\in C^1([0,\de_{i-1}^2],\M^s)$ as the solution of $\ddt g_i=P_{g_i}(k(u_{i-1},g_i))$ with $g_i(0)=g_0$;
\item $u_i\in C^{2,1,\alpha}(M_{\de_i})$ as the solution of $\pt u_i=\tau_{g_i}(u_i)$, $u_i(0)=u_0$, defined and smooth on a maximal domain $M_{\de_i}:=M\times [0,\de_i^2)$, $\de_i\leq \de_{i-1}$.
\end{itemize}
Here we use the Lipschitz-continuity of the map $P_g$ on the Banach manifold $\M^s$ in the first step. We also remark that the equation for $u_i$ is a semilinear parabolic equation so  
standard methods, see e.g.~\cite{Lin-Wang} Theorem 5.2.1, lead to the existence of a solution $u_i$ of the above equation, defined on all of $[0,\de_{i-1}^2)$ unless there is a blow-up in the gradient 
at some time $\de_i^2$, $0<\de_i<\de_{i-1}$. 

We claim that for $\de_0$ initially chosen small enough, the iterates are all defined on $[0,\de_0^2)$ and satisfy 
\beqa \label{est:iteration}
\norm{u_{i+1}-u_{i}}_{C^{2,1,\alpha}(M_{\de_0})}^*&\leq \frac12\norm{u_i-u_{i-1}}_{C^{2,1,\alpha}(M_{\de_{0}})}^*\\
\norm{g_{i+1}-g_{i}}_{C^{1}([0,\de_0^2],\M^s)}^*&\leq C\cdot \de_0 \norm{u_{i+1}-u_{i}}_{C^{2,1,\alpha}(M_{\de_0})}^*,\\
\eeqa
thus converging to a classical solution $(u,g)\in C^{2,1,\alpha}([0,\de_0^2)\times M)\times C^1([0,\de_0^2),\M^s)$ in the limit $i\to \infty$.
Here, we use scaling invariant versions of the standard parabolic H\"older norms, defined by 
$\norm{u}_{C^{0,\alpha}(M_\delta)}^*=\norm{u}_{C^{0}(M_\delta)}+\delta^\alpha\seminorm{u}_{C^{\alpha}(M_\delta)} $ and more generally
$$\norm{u}_{C^{a, b,\alpha}(M_\delta)}^*=\sum_{\substack{k+2j\leq a\\ j\leq b}}\delta^{2j+k}\norm{\pt^j\nabla^k u}_{C^{\alpha}(M_\delta)}^*,\qquad M_\de=[0,\de^2)\times M.$$

We remark that the second estimate of \eqref{est:iteration} immediately follows from Propostion \ref{Prop:S} and the Gronvall lemma, compare with \eqref{est:g1}. 
To estimate $w_i=u_i-u_{i-1}$, we observe that 
\beq\label{eq:wi}
\pt w_i-L_iw_i=f_i\eeq
for the elliptic linear operator
$$L_i w:=\Delta_{g_i}w+A_{g_i}(u_{i-1})(\na u_{i-1},\na w)+\big(dA_{g_i}(u_{i-1})\big)(w)(\na u_{i-1},\na u_{i-1})$$
and a right hand side that is bounded in $C^{0,\alpha}(M_\de)$ for any $\de\leq \de_i$ by 
$$\de^2\norm{f_i}_{C^{0,\alpha}(M_\de)}^*\leq C \norm{g_i-g_{i-1}}_{C^1([0,\de^2],\M^s)}^*+C(\norm{w_i}_{C^{2,1,\alpha}(M_\de)}^*)^2$$
with a constant depending on a $C^{2,1,\alpha}$ bound on the previous iterate $u_{i-1}$ but not on $u_i$.
We then apply the following scaling invariant version of parabolic Schauder estimates
\begin{Prop}\label{Prop:parab}
Let $M$ be a closed manifold and let $\lambda>0$, $A<\infty$ be fixed. Then there exists a number $C<\infty$ such that the following holds true.
Let $L$ be any second order differential operator on  $M_\delta$, $\delta\in (0,1)$ any number, that is given in local coordinate charts as
$Lu=\dxi (a^{ij}\dxj u)+b^i \dxi u+cu$ with 
$$a^{ij}(x,t)\xi_i\xi_j\geq \lambda\abs{\xi}^2\, \text{ for all }\xi\in \R^m \text{ and } (x,t)\in M_\delta$$
$$\norm{a^{ij}}_{C^{1,0,\alpha}(M_\delta)}^*+\delta\norm{b^i}_{C^{\alpha}(M_\delta)}^*+\delta^2\norm{c}_{C^{\alpha}(M_\delta)}^*\leq A.$$
Then the solution $u\in C^{2,1,\alpha}(M_\delta)$ of 
$
\pt w-Lw=f\in C^\alpha(M_\delta), w(0)=0
$ satisfies 
\beqs \norm{w}_{C^{2,1,\alpha}(M_\delta)}^*\leq C\delta^2 \norm{f}_{C^{\alpha}(M_\delta)}^*.
\eeqs
\end{Prop}
In terms of giving a proof of this result, we remark that standard Schauder estimates combined with a scaling argument give in a first step an estimate of the form
\beq \norm{w}_{C^{2,1,\alpha}(M_\delta)}^*\leq C\delta^2 \norm{f}_{C^{\alpha}(M_\delta)}^*+C\norm{w}_{C^{0}(M_\delta)}\label{est:parabolic}\eeq
for constants $C$ independent of $\de$. To retain this scaling invariance, we can then use the Ehrling-lemma and a further rescaling argument to estimate the second term of \eqref{est:parabolic} by 
\beqs
\norm{w(t)}_{C^0(M)}\leq \eps\cdot \norm{w}_{C^{\alpha}(M_\delta)}^*+C_\eps \sup_{x\in M} \delta^{-1}\norm{w(t)}_{L^2(B_\delta(x))}
\eeqs
on every time slice. Finally considering the evolution of local energy quantities of the form  $\int\varphi^2\big[a^{ij}\partial_{x_i}w(t)\partial_{x_j}w(t)+\delta^{-2}w(t)^2\big]\,dx$,
$\varphi$ a cut-off function supported on balls of radius $2\de$, gives that the last term in this estimate is bounded by a fixed (independent of $\de$) multiple of $\de^2\norm{f}_{L^\infty}$, completing the proof of Propostion \ref{Prop:parab}.

Turning back to the equation \eqref{eq:wi} satisfied by $w_{i}$, this Schauder-estimate allows us to conclude that for any $\de<\de_i$ 
$$\norm{w_i}_{C^{2,1,\alpha}(M_\de)}^*\leq C\de\norm{w_{i-1}}_{C^{2,1,\alpha}(M_\de)}^*+C\big(\norm{w_i}_{C^{2,1,\alpha}(M_\de)}^*\big)^2.$$
Since $w_i(0)=0$, the norm $\norm{w_i}_{C^{2,1,\alpha}(M_\de)}^*$ is small at least for $\de$ small (a priori depending on $i$). We conclude that the first estimate of 
\eqref{est:parabolic} holds true, initially for $\de$ small and then, by a continuity argument, indeed for as long as the solution exists (provided $\de_0=\de_0(u_0,g_0)$ was initially chosen small enough). 
But this very estimate prevents a blow-up before time
$\de_{i-1}$, so that $\de_i=\de_{i-1}=..=\de_0$, completing the proof.

\subsection{Proof of Lemma \ref{lemma:pull-back}}
We finally provide a possible proof of the fact that any metric in $\M^s$ can be written in the form $f^*\bar g$ with $f\in \D^{s+1}$ and $\bar g\in \M$.

Let $s>3$ and let $\Omega\subset \M^s$ be the subset of all metrics which can be written in the form $g=f^*\bar g$ for a smooth metric $\bar g\in \M$ and a diffeomorphism $f$ of class $H^{s+1}$. 

We first prove that $\Omega$ is an open subset of $\M^s$ using the slice theorem. Given any metric of the form $g_0=f_0^*\tilde g_0$, $f_0\in \D^{s+1}$ and $\tilde g_0\in \M$ we apply the slice-theorem \ref{thm:Tromba} to the smooth 
metric $\tilde g_0$ resulting in an $\M^s$-neighbourhood $\tilde W$ of $\tilde g_0$, consisting only of metrics of the form $f^*g_S$, 
$g_S$ an element of a slice $S$ around $\tilde g_0$ and thus in particular smooth.
The pull-back $W=f_0^* \tilde W$ is then an $\M^s$-neighbourhood of the original metric $g_0\in\M^s$, containing only metrics of the form 
$g=f_0^* (f^*g_S)=(f\circ f_0)^* g_S$, $g_S\in S\subset \M$ and $f\circ f_0\in \D^{s+1}$. So indeed $W\subset \Omega$ and $\Omega$ is open.

To see that $\Omega$ is also closed, we use a result due to Ebin and Palais which says that the action of $\D^{s+1}$ on $\M^s$ is proper, see e.g.~Theorem 2.3.1 in \cite{Tromba}; 
in practice this means that if we are given a sequence of diffeomorphisms $f_i\in \D^{s+1}$ and a convergent sequence of metrics $g_i\to g$ in $\M^s$ 
then knowing that $f_i^*g_i\to \bar g\in \M^s$ converges (in $H^s$ topology) is enough to conclude that also (a subsequence of) the diffeomorphisms $f_i$ converges, 
$f_i\to f$ in $\D^{s+1}$.

Let now $g\in \M^s$ be such that there are diffeomorphisms $f_i\in \D^{s+1}$ and metrics $g_i\in\M$ with $f_i^*g_i\to g$ (in $\M^s$). This convergence implies in particular that the 
length $\ell(g_i)=\ell(f_i^*g_i)$ of the shortest closed geodesic of $(M,g_i)$ is bounded away from zero. Thus
the Mumford compactness theorem implies that after pulling back $g_i$ by a \textit{smooth} family of diffeomorphisms $\tilde f_i$, a subsequence of $g_i$ converges smoothly 
$$(\tilde f_i)^*g_i=(f_i^{-1}\circ \tilde f_i)^*(f_i^*g_i)\to \bar g\in \M.$$
We conclude that the diffeomorphisms $f_i^{-1}\circ \tilde f_i$ converge to another diffeomorphism $f\in \D^{s+1}$ and thus that 
$g=(f^{-1})^*\bar g\in \Om$.

{\sc Max-Planck-Intitut f\"ur Gravitationsphysik, Am M\"uhlenberg 1, \\14476 Potsdam, Germany}
	

\begin{thebibliography}{10}

\bibitem{Ding-Li-Liu}
W.~Ding, J.~Li, and Q.~Liu.
\newblock Evolution of minimal torus in riemannian manifolds.
\newblock {\em Invent. Math.}, 165:225--242, 2006.

\bibitem{Ding-Tian}
W.-Y. Ding and G.~Tian.
\newblock Energy identity for a class of approximate harmonic maps from
  surfaces.
\newblock {\em Comm. Anal. Geom.}, 3:543--554, 1995.

\bibitem{Eells+Sampson}
J.~Eells and H.~J. Sampson.
\newblock Harmonic mappings of riemannian manifolds.
\newblock {\em Amer. J. Math.}, 86(1):109--160, 1964.

\bibitem{Freire2}
A.~Freire.
\newblock Uniqueness of the harmonic map flow from surfaces to general targets.
\newblock {\em Comm. Math. Helv.}, 70(1):310--338, 1995.

\bibitem{Freire}
A.~Freire.
\newblock Uniqueness of the harmonic map flow in two dimensions.
\newblock {\em Calc. of Var.}, 3(1):95--105, 1995.

\bibitem{GOR}
R.~D. Gulliver, R.~Osserman, and H.~L.Roydon.
\newblock A theorey of branched immersions of surfaces.
\newblock {\em Amer. J. Math.}, 95:750--812, 1973.

\bibitem{Kob-Nom}
S.~Kobayashi and K.~Nomizu.
\newblock {\em Foundations of Differential Geometry}.
\newblock Interscience Publishers, 1963.

\bibitem{Lieberman}
G.~Lieberman.
\newblock {\em Second order parabolic differential equations}.
\newblock World Scientific, 1996.

\bibitem{Lin-Wang}
F.~Lin and C.~Wang.
\newblock {\em The analysis of harmonic maps and their heat flows}.
\newblock World Scientific, 2008.

\bibitem{Rupflin}
M.~Rupflin.
\newblock An improved uniqueness result for the harmonic map flow in two
  dimensions.
\newblock {\em Calc. of Var.}, 33(3):329--341, 2008.

\bibitem{R-T}
M.~Rupflin and P.~Topping.
\newblock Flowing maps to minimal surfaces.
\newblock {\em arXiv:1205.6298v1}.

\bibitem{Rup-Top-Zhu}
M.~Rupflin, P.~Topping, and M.~Zhu.
\newblock Decomposition of maps into minimal immersions.
\newblock {\em in preparation}.

\bibitem{Sacks-Uhlenbeck}
J.~Sacks and K.~Uhlenbeck.
\newblock Minimal immersions of closed riemann surfaces.
\newblock {\em Trans. Amer. Math. Soc.}, 271:639--652, 1982.

\bibitem{Schoen-Yau}
R.~Schoen and S.~T. Yau.
\newblock Existence of incompressible minimal surfaces and the topology of
  three dimensional manifolds with non-negative scalar curvature.
\newblock {\em Annals of Math.}, 110:127--142, 1979.

\bibitem{Struwe85}
M.~Struwe.
\newblock On the evolution of harmonic mappings of riemannian surfaces.
\newblock {\em Comment. Math. Helv.}, 60:558--581, 1985.

\bibitem{Topping}
P.~Topping.
\newblock Reverse bubbling and nonuniqueness in the harmonic map flow.
\newblock {\em Int. Math. Research Notices}, 10:558--581, 2002.

\bibitem{Tromba}
Tromba.
\newblock {\em Teichm\"uller theory in Riemannian geometry}.
\newblock Lectures in Mathematics ETH-Z\"urich. Birkh\"auser, 1992.

\end{thebibliography}
\end{document}